\documentclass[11pt]{article}

\usepackage{amssymb}
\usepackage[centertags]{amsmath}
\usepackage{amsthm}

\newtheorem{theorem}{Theorem}
\newtheorem{proposition}{Proposition}

\newtheorem{lemma}{Lemma}
\newtheorem{example}{Example}

\newcommand{\co}{\overline{co}\hbox{ }}
\newcommand{\R}{\mathbb{R}}
\newcommand{\skal}[2]{\langle #1,#2 \rangle}
\newcommand{\sign}{\operatorname{sign}}

\def\dss{\displaystyle}

\title{Existence of Solutions for \\ Nonconvex Differential Inclusions of Monotone Type}
\author{Elza Farkhi\thanks{School of Mathematical Sciences, Sackler Faculty of Exact
Sciences, Tel Aviv University, 69978 Tel Aviv, Israel, email: 
elza@post.tau.ac.il} \quad Tzanko Donchev\thanks{Department of Mathematics, "Al. I.
Cuza" University, Ia\c{s}i 700506, Romania, email: tzankodd@gmail.com}
 \quad Robert Baier\thanks{University of Bayreuth, Department of Mathematics,
Chair of Applied Mathematics, 95440 Bayreuth, Germany, email: 
robert.baier@uni-bayreuth.de}  
}

\begin{document}

\maketitle


\begin{abstract}
Differential inclusions with compact, upper semi-continuous,  not necessarily convex right-hand sides 
in $\mathbb{R}^n$ are studied. 
Under a weakened monotonicity-type condition the 
existence of solutions is proved. 
\end{abstract}
\begin{quote}
\end{quote}
\begin{quote}
\end{quote}

\section{Introduction} 

We study the autonomous differential inclusion: 
\begin{equation}\label{1} 
 \dot x(t)\in F(x(t)),\quad x(0)= x_0\in \mathbb{R}^n,\quad t\in I=[0,T], 
\end{equation} 
where the set-valued mapping $F$ has compact, not necessarily convex values in $\mathbb{R}^n$,
and is upper semi-continuous, or equivalently, has a closed graph. We also assume
linear growth of $F$, to ensure boundedness of all solutions, 
and a weakened 
monotonicity-type condition in the spirit of the strengthened one-sided Lipschitz (S-OSL) condition
\cite{LemVel:98}.

The results on existence of solutions of such inclusions 
are not so numerous. 
First, one should
mention the well-known existence result in the case of maximal monotone right-hand sides \cite[Sec.~3.2, Theorem~1]{AubCel:84}.
Maximal monotone set-valued maps, as is well-known, are almost everywhere single-valued \cite{Zar:73,Kend:74},
 and at the points where they are not single-valued, their values are convex sets.
Other important existence results for differential inclusions with non-convex right-hand sides
 are the results of Filippov \cite{FIL:67} for Lipschitz $F$, and of Hermes \cite{Her:71}, 
who relaxed the Lipschitz continuity of $F$ to continuity with bounded variation.
The result of \cite{BreCelCol:89} is for upper semi-continuous and cyclically monotone map $F$,
which is a stronger condition than just monotonicity. In \cite{KrRibTsa:07} the phenomenon of
``colliding'' on the set of discontinuities of $F$ is studied and conditions to avoid 
or to escape from this set are investigated.

We prove the existence under another monotonicity-type condition that ensures 
componentwise
monotonicity of the Euler polygons and their derivatives, which is the key for this existence proof.
The meaning of this condition is that the set-valued map $-F(\cdot)$ 
(with images being the pointwise negation of $F(x)$)
satisfies the strengthened one-sided Lipschitz condition \cite{LemVel:98} with a constant zero.
The latter condition is a weaker form
of the S-OSL condition for set-valued maps introduced  in~\cite{Lem:93}, see~\cite[Remark~2.1]{LemVel:98}.

 We give examples that show that our condition, although simple, does not imply monotonicity,
hence does not imply cyclical monotonicity.

\section{Main result}

First we introduce some notation. For every notion used in the paper, but not 
explicitly defined here we refer the reader to \cite{D} . 

Let $v\in \mathbb{R}^n$. We denote by $|v|$ the Euclidean norm of the vector $v$
and by $v_j$ its $j$--th coordinate, i.e.~$v= (v_1,v_2,\ldots,v_n)$. 
Denote by $\mathbb{B}$ the unit ball in $\mathbb{R}^n$. For a bounded set $A\subset\R^n$,
we denote $\|A\|=sup\{\|a\|:a\in A\}$.

We impose the following assumptions in order to prove the existence of solution: 
\vskip 0.5em 
\textbf{A1.} $F: \mathbb{R}^n\rightrightarrows \mathbb{R}^n$ has
compact, nonempty values and closed graph.
\vskip 0.5em 
\textbf{A2.} \textbf{\emph{Linear growth condition}}
There exist constants $A$ and $B$ such that $\| F(x)\|\leq A+ B|x|$
for any $x\in \mathbb{R}^n$. 

The following lemma is a corollary of Gronwall inequality and
{\bf A2}, and  its proof 
is given in~\cite[Remark~3.1]{DonFar:2000} (see also \cite{DonFar:98}). 
\begin{lemma} \label{l1}
 Under \textbf{A1, A2} there exist constants $L$ and $M$ such that $|x(t)|\leq L$ 
and $|\dot x(t)|\leq M$ for every solution $x(\cdot)$ of 
\[
\dot x(t)\in \co F(x(t)+ \mathbb{B})+ \mathbb{B},\ x(0)= x_0. 
 \]
\end{lemma}
\vskip 0.5em 
\textbf{A3.} \textbf{\emph{Weak Componentwise Monotonicity (WCM) Condition}}: For every $x,y\in\mathbb{R}^n$ and every 
$v\in F(x)$ there exists $w\in F(y)$ such that 
\begin{equation}\label{mc} 
 (x_j- y_j)(v_j- w_j)\geq 0,\quad \forall \ j=1,2,\ldots, n. 
\end{equation}
In other words, \eqref{mc} means that the negation of the given set-valued map,  $-F(\cdot)$ 
satisfies the S-OSL condition from \cite{LemVel:98} with a constant zero. 

\begin{theorem}\label{ThM}
 Under the conditions \textbf{A1, A2, A3} the differential inclusion (\ref{1}) has a 
solution. 
\end{theorem} 
 To proof the theorem, we use the following Euler-Cauchy construction of polygonal approximate solutions. Fix 
the natural number $N$ and let the mesh size  $\dss h=\frac{T}{N}$ be such that $hM< 1$. Denote 
the mesh points by $t_i= ih$. We define Euler's polygons $x^N:[0,T]\to \R^n$ in the following way: 
We set $x^N(0)=x_0$, and for $t\in [0,t_1]$, we construct 
$x^N(t)= x_0+ tv^0$, 
where $v^0\in F(x_0)$ is arbitrary. 
Further, we construct subsequently the Euler polygons in each subinterval \ $t\in [t_i,t_{i+1}]$, \  
for \ $i=1,\ldots, N-1$, \ by \ $x^N(t)= x^N(t_i)+ (t-t_i)v^i$, \
where the velocity \ $v^i\in F(x_N(t_i))$ is chosen by the assumption A3, such that 
\[
(x^N_j(t_i)- x^N_j(t_{i-1}))(v^i_j- v^{i-1}_j)\geq 0,\ j=1,\ldots,n. 
\]
The following lemma and proposition 
represent the main steps of the proof of Theorem \ref{ThM}.
\begin{lemma}\label{mon}
The polygonal functions $x^N_j(t)$ and their derivatives  $\dot x^N_j(t)$ are monotone 
for every \\ $j\in \{1,2,...,n \}$.
\end{lemma}
\begin{proof} Fix a coordinate $j\in \{1,...,n\}$, and suppose that \ $v^i_j=0$ for $i<k$  and $v^k_j\neq 0$.
Here $k=0$ is possible, i.e.~possibly $v_0^j\neq 0$. Clearly, if  $v^i_j=0$ for all $i\le N$, 
then the claim holds trivially.   If $v^k_j>0$,
then $x^N_j(\cdot)$ is strictly monotone increasing on the subinterval $[t_k,t_{k+1}]$, and therefore 
$x^N_j(t_{k+1})> x^N_j(t_k)$. Again, using the assumption (\ref{mc}),
it is easy to see that $v_j^{k+1}\ge v_j^k>0$. 
On the next subintervals, $[t_i,t_{i+1}],\ i>k$, 
continuing in the same way, we show that  
$\dss \{ v^i_j\}_{i=k}^\infty$ 
is positive and monotone nondecreasing, while $x^N_j(t)$ is increasing. 
If for some $j\in \{1,...,n\}$, \ $v^i_j=0$ for all $i<k$, and $v^k_j<0$, then in a similar way we get that 
$\dss \{ v^i_j\}_{i=k}^\infty$ is negative and monotone nonincreasing, while $x^N_j(t)$  is strictly 
monotone decreasing.  
\end{proof}

The following proposition 
is proved using Helly's selection principle \cite[Chap.~10]{KF} replacing the 
Arzel\`{a}-Ascoli theorem, which is usually applied in precompactness proofs for continuous functions, and used 
to prove the existence of solutions for differential
inclusions with convex right hand sides (see e.g.~\cite[Theorem~2.2]{DonLem:92}).
\begin{proposition}\label{exist} Under the conditions \textbf{A1, A2, A3},
the sequence $x^N(\cdot)$ has a subsequence converging uniformly on $I$ to a 
function $x^\infty(\cdot)$, 
with each coordinate $x_j^\infty(\cdot)$ being monotone. 
Furthermore, $x^\infty(\cdot)$ is a solution of the inclusion \eqref{1}.
\end{proposition}


Let us recall that a mapping $F:\R^n\rightrightarrows \R^n$ is \emph{monotone} if
 \begin{align}
    \skal{x - y}{v - w} & \geq 0 \quad (\text{for all}\ x,y \in \R^n,\ v \in F(x), \ w \in F(y)).
   \label{eq:monotone}
 \end{align}
The map $F:\R^n\rightrightarrows \R^n$ is \emph{cyclically monotone} if for every cyclic sequence of 
points  \ $x_0,x_1,\ldots,x_N=x_0$ and all \ $v_i\in F(x_i), i=1,\ldots,N$,
 \begin{align}
    \sum_{i=1}^N\skal{x_i - x_{i-1}}{v_i} & \geq 0. 
   \label{eq:cycl_monotone}
 \end{align}
It is easy to check that every cyclically monotone map is monotone.
The classical monotonicity condition \eqref{eq:monotone} requires that $F(\cdot)$ is almost everywhere single-valued
\cite{Zar:73, Kend:74}. 

In~\cite{BreCelCol:89} an existence proof for solutions of differential inclusions with compact right-hand 
side is given which is cyclically monotone. It is also proved that cyclically monotone map have images
that are subsets of a subdifferential map of a convex function.  

\section{Examples}

We give here examples of set-valued maps 
which are weakened 
monotone and fulfill \textbf{A3},
but are neither monotone nor cyclically monotone.

The mappings of the examples below are not monotone, hence are not cyclically monotone,
 since they are not single-valued almost everywhere.

The following example is a modification of~\cite[Example~2.1]{DonFar:98} in which $G(x) = -F(x)$ is shown to be
 OSL. Here, $F(\cdot)$ satisfies A3, but is not monotone and is discontinuous. 

\begin{example}
   Let $F: \R \rightrightarrows \R$ be defined as
  \label{ex:discont_osl_conv}
    \begin{align*}
       F(t) & = \begin{cases}
                    [-1,0] & \quad (t < 0) \,, \\
                    [-1,1] & \quad (t \geq 0) \,.
                \end{cases}
    \end{align*}
   Then, $F(\cdot)$ has convex images and satisfies \textbf{A3}, but is not
   monotone in the sense of~\eqref{eq:monotone}.
\end{example}

There are maps with compact images fulfilling \textbf{A3} that are weakened monotone, but not monotone,
as the following example 
shows.

\begin{example}
   Let $F,G: \R \rightrightarrows \R$ be defined as $F(t)=\{ t, t^{\frac{1}{3}}\}, \ G(t) =\{ t^{\frac{1}{3}}, t+\sign(t)\}$. 
   Then, $F(\cdot), G(\cdot)$ have compact non-convex images, $F$ is continuous, while $G$ is discontinuous at the origin. Both $F$ and $G$
	satisfy \textbf{A3}, but are not monotone in the sense of~\eqref{eq:monotone}.
  \label{ex:cont_osl_comp}
\end{example}

To construct examples of set-valued maps that satisfy \textbf{A3} in higher dimensions, we may take
the Cartesian product 
of such one-dimensional mappings 
and use the simple fact that the union of two
mappings that satisfy \textbf{A3} also satisfy \textbf{A3}.
\begin{example}  \label{ex:2d}
  Let $F: \R^2 \rightrightarrows \R^2$ be defined as
    \begin{align*}
       F(x) & = \begin{cases}
                    (\{x_1^{\frac{1}{3}}\} + [-1,0]) \times ([-2,-1] \cup [1,2]) & \quad (x_1 < 0) \,, \\
                    (\{x_1^{\frac{1}{3}}\} + [-1,1]) \times ([-2,-1] \cup [1,2]) & \quad (x_1 \geq 0) 
                \end{cases}
    \end{align*}
   for $x = (x_1,x_2) \in \R^2$. \\
   Then, $F(\cdot)$ has compact images and satisfies \textbf{A3}, but is not
   monotone in the sense of~\eqref{eq:monotone}.
\end{example} 

\begin{example} \label{multdim}
Let $f: \mathbb{R}\rightrightarrows \mathbb{R}$ be defined as follows: 
\begin{align*} 
f(s)=   \begin{cases} 
 \dss \left\{\sign(s) \right\} & s\neq 0. \\ 
   \{ -1, 1\}  & x=0. 
 \end{cases}
\end{align*} 
Define $F: \mathbb{R}^n \rightrightarrows \mathbb{R}^n$ by 
$\dss  F(x)= \{ \frac{1}{2} \big(f(x_1),f(x_2),\ldots, f(x_n) \big), \ 
\big(f(x_1),f(x_2),\ldots, f(x_n) \big) \}$. 
Clearly, $F(\cdot)$ satisfies all our conditions, but  is neither 
monotone, nor cyclically monotone. 
\end{example}
Clearly, there are monotone mappings which do not satisfy \textbf{A3}. A simple example is the subdifferemtial
of the Euclidean norm. We believe that there are other classes  set-valued maps of monotone-type for which
 existence of solutions of differential inclusions with non-convex 
upper semi-continuous right-hand sides can be proved.

\vspace{0.5em}

\textbf{Acknowledgements.} 
The research of the first and the third authors is partially supported by Minkowski Minerva 
Center for Geometry at Tel-Aviv University. 
The second author is supported by a grant of the Romanian National Authority
for Scientific Research, CNCS-UEFISCDI, project number PN-II-ID-PCE-2011-3-0154.
The third author is partially supported also by the European Union Seventh Framework 
Programme [FP7-PEOPLE-2010-ITN] under grant agreement 264735-SADCO.

\end{document}